\documentclass[12pt]{article}

\usepackage[english]{babel}
\usepackage{cite}
\usepackage{latexsym}
\usepackage{amssymb,amsthm,amsmath}
\usepackage{hyperref}
\usepackage{graphicx}
\usepackage{longtable}
\usepackage{xcolor}
\usepackage{enumerate}
\usepackage{mathdots}
\usepackage{comment}
\usepackage{cancel}
\usepackage[normalem]{ulem}

\newtheorem{theorem}{Theorem}[section]

\newtheorem{lemma}[theorem]{Lemma}
\newtheorem{corollary}[theorem]{Corollary}

\newtheorem{result}[theorem]{Result}

{\theoremstyle{definition}
\newtheorem*{definition*}{Definition}
}
\newtheorem{proposition}[theorem]{Proposition}
\newtheorem*{corollary*}{Corollary}
\newtheorem*{lemma*}{Lemma}
\newtheorem*{remark*}{Remark}

\newcommand{\F}{\mathbb {F}}

\newcommand{\N}{\mathrm {N}}
\newcommand{\Tr}{\mathrm {Tr}}

\def\cD{\mathcal D}
\def\cC{\mathcal C}
\def\cS{\mathcal S}

\definecolor{dgreen}{rgb}{0.13,0.7,.63}

\DeclareMathOperator{\PG}{{PG}}
\DeclareMathOperator{\AG}{{AG}}
\newcommand{\la}{\langle}
\newcommand{\ra}{\rangle}

\usepackage{authblk}

\title{On a conjecture about maximum scattered subspaces of $\mathbb{F}_{q^6}\times \mathbb{F}_{q^6}$\vspace{.3cm}}

\author[$1$]{Daniele Bartoli \footnote{Corresponding author.}}
\author[$2$]{Bence Csajb\'ok \footnote{Present address: Dipartimento di Meccanica, Matematica e Management, Politecnico di Bari, Via Orabona 4, 70125 Bari, Italy}}
\author[$3$]{Maria Montanucci}

\affil[$1$]{Department of Mathematics and Informatics, University of Perugia, Perugia,  Italy,  \textit{daniele.bartoli@unipg.it}\vspace*{.3cm}}

\affil[$2$]{ELKH--ELTE Geometric and Algebraic Combinatorics Research Group,
	ELTE E\"otv\"os Lor\'and University, Budapest, Hungary, Department of Geometry, 
	1117 Budapest, P\'azm\'any P.\ stny.\ 1/C, Hungary, \textit{csajbokbence@caesar.elte.hu} \vspace*{.3cm}}

\affil[$3$]{Department of Applied Mathematics and Computer Science, Technical University of Denmark, Kongens Lyngby, Denmark, \textit{marimo@dtu.dk}}

\date{}

\begin{document}
\maketitle

\begin{abstract}
We prove that under the action of $\mathrm{GL}(2,q^6)$ there are $\lfloor (q^2+q+1)(q-2)/2 \rfloor$ equivalence classes of maximum scattered subspaces of the form $U_b=\{(x,bx^q+x^{q^4}) : x\in \F_{q^6}\}$ in $\F_{q^6}\times \F_{q^6}$. This verifies a conjecture of Csajb\'ok, Marino, Polverino and Zanella from 2018.
\end{abstract}

\medskip

\noindent
{\bf MSC:} 51E20; 05B25; 51E22\\
{\bf Keywords:} Scattered subspaces; Linear sets; MRD-codes
\section{Introduction}

Maximum scattered subspaces are not only objects of intrinsic interest in finite geometry but also powerful tools for the construction of MRD-codes, projective two-weight codes, and strongly regular graphs. In \cite{CMPZ} Csajb\'ok, Marino, Polverino, and Zanella introduced a new family of maximum scattered subspaces $U_b=\{(x,f_b(x)) : x\in \F_{q^6}\}$ in $\F_{q^6} \times \F_{q^6}$ arising from polynomials of type $f_b(x)=bx^q+x^{q^4}$ for certain choices of $b \in \mathbb{F}_{q^6}$. 
In this paper, first we find necessary and sufficient conditions for $b$ to obtain a maximum scattered subspace. Then we prove \cite[Conjecture 7.5]{CMPZ} on the number of new and inequivalent maximum scattered subspaces of this family.

For a prime power $q$ let $V$ denote an $r$-dimensional vector space over the finite field $\F_{q^n}$.
Then $V$ is also an $\F_q$-vector space of dimension $rn$. A $t$-spread of $V$ is a partition of $V\setminus \{{\bf 0}\}$ by $\F_q$-subspaces of dimension $t$. In particular 
$\cD:=\{ \la {\bf u} \ra_{\F_{q^n}}\setminus \{{\bf 0}\} : {\bf u} \in V\}$ is the so called Desarguesian $n$-spread of $V$. An $\F_q$-subspace $U$ of $V$ is called scattered with respect to a spread $\cS$ if $U$ meets each element of $\cS$ in at most a one-dimensional $\F_q$-subspace. 
In \cite{BL2000} Blokhuis and Lavrauw proved that $rn/2$ is the maximum dimension of a scattered subspace of $V$ w.r.t. a Desarguesian spread. In this paper by a scattered subspace we will always mean a scattered subspace w.r.t. the Desarguesian spread $\cD$. When $rn$ is even then scattered subspaces of dimension $rn/2$ are called maximum scattered.  

Denote by $\Lambda$ the projective space $\PG(r-1,q^n)$ formed by the lattice of $\F_{q^n}$-subspaces of $V$. For a $k$-dimensional $\F_q$-subspace $U$ of $V$ the linear set of rank $k$ defined by $U$ is
\[ L_U = \{\langle {\bf u}\rangle_{\F_{q^n}} :  {\bf u} \in  U \setminus\{{\bf 0}\}\}\subseteq \Lambda.\]
When dealing with linear sets, equivalence is an important issue. 
Two linear sets $L_U$ and $L_W$ of $\Lambda$ are called $\mathrm{P}\Gamma \mathrm{L}$-equivalent
if there exists an element $\psi\in \mathrm{P}\Gamma \mathrm{L}(r, q^n)$ such that $\psi(L_{U}) =L_W$. It is possible that   $L_U$ and $L_W$ are $\mathrm{P}\Gamma \mathrm{L}$-equivalent  even if $U$ and $W$ are not $\Gamma \mathrm{L}$-equivalent, i.e. there is no $\varphi \in \Gamma \mathrm{L}(r,q^n)$ such that $\varphi(U)=W$; see for instance \cite{5,8}. 

If $U$ is a (maximum) scattered subspace, then $L_U$ is called a (maximum) scattered linear set. Maximum scattered linear sets of $\Gamma$ are two-intersection sets w.r.t. hyperplanes: they meet hyperplanes of $\Gamma$ either in $(q^{\frac{rn}{2}-n}-1)/(q-1)$ or in $(q^{\frac{rn}{2}-n+1}-1)/(q-1)$ points; see \cite{BL2000}. It follows that they define projective two-weight codes and strongly regular graphs; see \cite{c20}. 
After the series of papers \cite{2,3,BL2000,6,18} it is now known that scattered subspaces of dimension $rn/2$ exist whenever $rn$ is even. 

We briefly explain why maximum scattered subspaces of $\F_{q^n}^2$ are of particular interest. First of all, if $U_i$ is a maximum scattered subspace in $V_i=\F_{q^n}^{r_i}$, for $i=1,2$, then $U_1\oplus U_2$ is a maximum scattered subspace in $V_1\oplus V_2$, see \cite{3}. Thus, maximum scattered subspaces of  $\F_{q^n}^2$ can be used to construct maximum scattered subspaces in $\F_{q^n}^r$ for $r\geq 4$ as well. A maximum scattered $\F_q$-subspace of $\F_{q^n}^2$ is equivalent to a subspace $\{(x,f(x)) : x \in \F_{q^n}\}$ where $f(x)$ is a $q$-polynomial and there are $(q^n-1)/(q-1)$ directions determined by the graph of $f(x)$. By a result of Ball, Blokhuis, Brouwer, Storme, and Sz\H{o}nyi \cite{B2003, BBBSSz} this is the maximum number of directions that can be determined by the graph of an $\F_{q^n} \rightarrow \F_{q^n}$ function which determines less than $(q^n+3)/2$ directions and meets at least one line of $\AG(2,q^n)$ in at least $q$ points. Sheekey in \cite[Section 5]{Sheekey} discovered a  relation between maximum rank distance codes and maximum scattered subspaces of $\F_{q^n}^2$. For more details we refer also to \cite[Section 6]{CMPZ} and the survey \cite{SheekeySurvey}. Maximum scattered $\F_2$-subspaces of $\F_{2^n}^2$ correspond to translation ovals and they have been classified by Payne \cite{Payne}, see also \cite{Hirs0,Hirsbook} by Hirschfeld. Up to $\Gamma \mathrm{L}$-equivalence these subspaces are of the form $\{(x,x^{2^s}):x\in \F_{2^n}\}$ with $\gcd(s,n)=1$. This example can be generalized for every $q$, but for $q>2$ there are other examples of maximum scattered subspaces.
In this paper our main result is the following.

\begin{theorem}
	\label{mainthm}
	Under the action of $\mathrm{GL}(2,q^6)$ there are $\lfloor (q^2+q+1)(q-2)/2 \rfloor$ equivalence classes of maximum scattered subspaces of the form $U_b=\{(x,bx^q+x^{q^4}) : x\in \F_{q^6}\}$ in $\F_{q^6}\times \F_{q^6}$.
\end{theorem}

At first, we will find necessary and sufficient conditions for $b$ to obtain a scattered subspace. Such conditions were find independently, with different techniques, also by Polverino and Zullo in \cite{PZ}. In this paper, we will follow the same notation and approach as in \cite[Section 7]{CMPZ}. 
We have, in particular, that $U_b$ is scattered if and only if for each $m \in \mathbb{F}_{q^6}$ the $q$-polynomial $r_{m,b}(x):=mx+bx^q+x^{q^4}$, seen  as an $\F_q$-linear transformation of $\F_{q^6}$, has kernel of dimension at most $1$ over $\F_q$. This is equivalent to require that the associated Dickson matrix $M_0(m,b)$, see \eqref{Matrix_M}, has rank at least $5$. It is usually difficult to determine the rank of this matrix. In \cite{CMPZ} the authors proved that, for certain values of $b$, certain $5\times 5$ minors of $M_0(m,b)$ cannot vanish at the same time. It then followed that, for these values, $U_b$ is a scattered subspace. At that time, it was not clear whether the vanishing of some $5\times 5$ minors could imply that all of these minors vanish. Thus, while this technique was efficient to prove that certain values of $b$ yields a scattered subspace, it was not enough to find all of them. The main result of \cite{qres} implies that the vanishing of $\det M_0(m,b)$ and the vanishing of a particular $5\times 5$ minor together force $M_0(m,b)$ to have rank less than $5$. In this way, we obtain minimal, sufficient and necessary conditions for $b$ to obtain a scattered subspace. This is done in Lemma \ref{mainlemma}, where we also use the MAGMA computer algebra system to calculate the resultant of certain multivariate polynomials. The final form of these conditions is obtained separately for $q$ odd, cf. Theorem \ref{q_odd}, and for $q$ even, cf. Theorem \ref{q_even}. As it turns out in the next section, it is enough to count the number of the values of  $N=b^{q^3+1}\in \F_{q^3}$, such that $U_b$ is scattered, in order to prove Theorem \ref{mainthm}. Since our conditions on $b$ involve the elementary symmetric polynomials of $N$ and its conjugates over $\F_{q^3}$, by using Vieta's formulas, our problem can be translated to count the number of certain degree $3$ polynomials. This will be done by using double counting arguments and by counting 
precisely the number of certain $\F_q$-rational points of some quadrics defined over $\F_q$.


\section{Preliminary results}

Recall that for $m \mid n$ the norm function 
$\N_{q^n / q^m} \colon \F_{q^n} \to \F_{q^m}$ is defined as $x \mapsto x^{(q^n-1)/(q^m-1)}$. 
The known maximum scattered $\mathbb{F}_q$-subspaces of $\F_{q^n}\times \F_{q^n}$ are members of one of the following families of scattered subspaces:
\begin{itemize}
\item $U_{s}^{1,n} \ :=\ \left\{\left(x, x^{q^s}\right) \ : \ x \in \mathbb{F}_{q^n}\right\}$, $1 \leq s \leq n -1$, $\gcd(s,n)=1$ \cite{BL2000,CZ2016};

\item $U_{s,\delta}^{2,n}  \ :=\ \left\{\left(x, \delta x^{q^s}+x^{q^{n-s}}\right) \ : \ x \in \mathbb{F}_{q^n}\right\}$, $n \geq 4$, $\N_{q^n/q}(\delta) \neq 0, 1$, $q\neq 2$, $\gcd(s,n)=1$, see \cite{LP2001} for $s=1$, \cite{LTZ2018,Sheekey} for $s\neq1$;

\item $U_{s,b}^{3,n}:=\left\{\left(x, b x^{q^s}+x^{q^{s+n/2}}\right) \ : \ x \in \mathbb{F}_{q^n}\right\}$, $n \in \{6, 8\}$, $\gcd(s,n/2)=1$, $\N_{q^n/q^{n/2}}(b) \neq 0,1$, for certain choices of $b$, see \cite{CMPZ};

\item $U_c^4:=\left\{\left(x, x^q+x^{q^3}+cx^{q^5}\right) \ : \ x \in \mathbb{F}_{q^6}\right\}$, $q$ odd, $c^2+c=1$, see \cite{7,trinomgen};

\item $U_h^5:=\left\{\left(x, h^{q-1}x^q-h^{q^2-1}x^{q^2}+x^{q^4}+x^{q^5}\right) \ : \ x \in \mathbb{F}_{q^6}\right\}$, $q$ odd, $h^{q^3+1}=-1$, see \cite{newscatt,vertex}.

\end{itemize}
Note that these families may overlap, but each of them contains examples not contained in other families. The only known families of scattered subspaces with examples for each $n>1$ are $U_s^{1,n}$ and $U_{s,\delta}^{2,n}$. By \cite{10}, for $n\leq 4$ every maximum scattered subspace is a member of one of these two families. In \cite{Zanellanuovo} it was proved that the condition $\N_{q^n/q}(\delta)\neq 1$ is not only sufficient but also necessary in order to get maximum scattered subspaces of the form $U_{s,\delta}^{2,n}$. 
Very recently in \cite{nuovo} it was proved that for $n\geq 10$ the family $U_{s,b}^{3,n}$ does not contain scattered subspaces. 
A polynomial $f(x)$ so that $\{(x^{q^t}, f(x)) : x \in \F_{q^{mn}}\}$ is scattered in $\F_{q^{nm}} \times \F_{q^{nm}}$ for infinitely many $m$ is called exceptional scattered of index $t$, see \cite{BM, BZ, FM}. The known exceptional scattered polynomials are those defining $U_s^{1,n}$ and $U_{s,\delta}^{2,n}$.

In this paper we will investigate the subspaces 
\[U_{s,b}:=U_{s,b}^{3,3}=\left\{\left(x, b x^{q^s}+x^{q^{s+3}}\right) \ : \ x \in \mathbb{F}_{q^6}\right\}\]
of $\F_{q^6}\times \F_{q^6}$. 
The first part of the next result is a special case of \cite[Proposition 5.1]{CMPZ}, the second part can be proved by the same arguments:

\begin{result}
\label{res0}
Take some $b,\bar{b} \in \F_{q^6}^*$, $\N_{q^6/q^3}(b) \neq 1$, $\N_{q^6/q^3}(\bar{b})\neq 1$ and $s,\bar{s}\in \{1,2,4,5\}$. 

\noindent
The subspaces $U_{s,b}$ and $U_{\bar{s}, \bar{b}}$ are $\Gamma \mathrm{L}(2,q^6)$-equivalent if and only if
\begin{itemize}
	\item $s=\bar{s}$ and $\N_{q^6/q^3}(\bar{b})=\N_{q^6/q^3}(b)^{\sigma}$, or
	\item $s+\bar{s}=6$ and $\N_{q^6/q^3}(\bar{b})\N_{q^6/q^3}(b)^{\sigma}=1$,
\end{itemize}
for some automorphism $\sigma \in Aut(\F_{q^3})$. 

\noindent
The subspaces $U_{s,b}$ and $U_{\bar{s}, \bar{b}}$ are $\mathrm{GL}(2,q^6)$-equivalent if and only if 
\begin{itemize}
	\item $s=\bar{s}$ and $\N_{q^6/q^3}(\bar{b})=\N_{q^6/q^3}(b)$, or
	\item $s+\bar{s}=6$ and $\N_{q^6/q^3}(\bar{b})^{q^{\bar{s}}}\N_{q^6/q^3}(b)=1$.
\end{itemize}
\end{result}

Also, by the first paragraph of \cite[Section 5]{CMPZ}, for $s\in \{1,2,4,5\}$ the $\F_q$-subspaces
\begin{equation}
\label{eki}
\{(x, \alpha x^{q^s}+\beta x^{q^{s+3}}) : x \in \F_{q^6}\} \quad \mbox{ and } \quad \{(x, \alpha' x^{q^s}+\beta' x^{q^{s+3}}) : x \in \F_{q^6}\}
\end{equation}
are $\mathrm{GL}(2,q^6)$-equivalent when $\N_{q^6/q^3}(\alpha \beta')=\N_{q^6/q^3}(\alpha' \beta)$. 
It follows that $U_{4,b}$ is equivalent to $U_{1,b}$ and $U_{2,b}$ is equivalent to $U_{5,b}$. 
Also, by Result \ref{res0}, $U_{5,b}$ is equivalent to $U_{1,b^{-q^5}}$ and hence the following holds. 

\begin{lemma}
Every subspace $U_{s,b'}$, $s\in \{1,2,4,5\}$, is $\mathrm{GL}(2,q^6)$-equivalent to a subspace
\[U_{b}=U_{1,b}=\{(x,bx^q+x^{q^4}) : x \in \F_{q^6}\},\]
for some $b$.
\end{lemma}

In \cite{CMPZ} the authors proved that $U_b$ is not $\Gamma \mathrm{L}(2,q^6)$-equivalent to subspaces of the form $U_{s}^{1,6}$ or  $U_{s,\delta}^{2,6}$; cf. \cite[Theorem 6.2]{CMPZ}.
By \cite{7,trinomgen} $U_b$ is not $\Gamma \mathrm{L}(2,q^6)$-equivalent to $U_c^4$ and by 
\cite[Proposition 3.5]{newscatt} neither to $U_h^5$. 
By Result \ref{res0} two subspaces $U_{b}$ and $U_{c}$,  $\N_{{q^{6}}/{q^3}}(b)\neq 1$, $\N_{{q^{6}}/{q^3}}(c)\neq 1$, are $\mathrm{GL}(2,q^{6})$-equivalent if and only if $\N_{{q^{6}}/{q^3}}(b)=\N_{{q^{6}}/{q^3}}(c)$ and they are $\Gamma \mathrm{L}(2,q^6)$-equivalent if and only if $\N_{{q^{6}}/{q^3}}(b)=\N_{{q^{6}}/{q^3}}(c)^\sigma$, for some $\sigma \in Aut(\F_{q^3})$. 
This motivates the definition of
\[\Gamma:=\{\N_{{q^{6}}/{q^3}}(b)\ : \ U_b \mbox{ is maximum scattered}\}.\]
The previous paragraph proves the following.

\begin{lemma}
The size of $\Gamma$ is the same as the number of pairwise not $\mathrm{GL}(2,q^6)$-equivalent maximum scattered subspaces of the form $U_{b}$. 	
\end{lemma}

In \cite{CMPZ} the authors proved that $\Gamma \neq \emptyset$ and conjectured that $|\Gamma|=\lfloor (q^2+q+1)(q-2)/2 \rfloor$. The conjecture was proven there for $q \leq 32$ by using GAP. In this paper we prove it for an arbitrary prime power $q$. For $q$ odd and even see Theorems \ref{Th:Conto} and \ref{qeven:Conto}, respectively. 

\medskip

Suppose $q=p^e$ for some prime $p$. From our proof, it will turn out that $U_b$ is maximum scattered if and only if $U_{b^p}$ is maximum scattered (cf. Theorems \ref{q_odd} and \ref{q_even}). It follows that $w \in \Gamma$ if and only if $w^p \in \Gamma$. Thus, a consequence of Theorem \ref{mainthm} is the following. 

\begin{corollary}
The number of pairwise not $\Gamma \mathrm{L}(2,q^6)$-equivalent maximum scattered subspaces of the form $U_b$ is at least
\[\frac{|\Gamma|}{3e}=\frac{\lfloor (q^2+q+1)(q-2)/2 \rfloor}{3e},\]
where $q=p^e$, for some $p$ prime.
\end{corollary}


\section{The proof of the conjecture}
Here we analyse more in detail the family $U_b$ in order to prove \cite[Conjecture 7.5]{CMPZ}. Doing so, we follow the same notation and approach as in \cite[Section 7]{CMPZ}. 
We have, in particular, that $U_b$ is scattered if and only if for each $m \in \mathbb{F}_{q^6}$ the $q$-polynomial $r_{m,b}(x):=mx+bx^q+x^{q^4}$, seen  as a linear transformation of $\F_{q^6}$, has kernel of dimension at most one over $\F_q$. In the whole section we will assume $b \neq 0$. 
This is equivalent to require that the associated Dickson matrix

\begin{equation}\label{Matrix_M}
M_0(m,b):=\left(
\begin{array}{cccccc}
m&b&0&0&1&0\\
0&m^q&b^q&0&0&1\\
1&0&m^{q^2}&b^{q^2}&0&0\\
0&1&0&m^{q^3}&b^{q^3}&0\\
0&0&1&0&m^{q^4}&b^{q^4}\\
b^{q^5}&0&0&1&0&m^{q^5}\\
\end{array}
\right)
\end{equation}
has rank at least five; for different proofs see \cite[Section 2.4]{25}, \cite[Proposition 5]{Menichetti}, \cite[Proposition 4.4]{WL}. As the definition of $\Gamma$ suggests, the norm of $b$ plays a crucial role in the investigation of $U_b$; in this whole section we will abbreviate it simply as
\begin{equation}
\label{N}
N:=\N_{q^6/q^3}(b)=b^{q^3+1}\in \F_{q^3}.
\end{equation}
Note that the determinant of the $4\times4$ top-right submatrix of $M_0(m,b)$ is $b^{q^2}(N-1)^{q}$ and it does not vanish when $b\neq 0 $ and $N\neq 1$. 
In order to compute the elements in $\Gamma$ we give a precise description of the values $b$ giving rise to a maximum scattered linear set $U_{b}$. We will need the following result.

\begin{theorem}[{\cite[Theorem 1.3]{qres}}]
	\label{Th:Sottobicchiere}
	Let $M(f)$ be the $n\times n$ Dickson matrix associated with a linearized polynomial $f$ over $\mathbb{F}_{q^n}$. Denote by $M_r(f)$ the $(n-r)\times (n-r)$ submatrix of $M(f)$ obtained by considering the last $n-r$ columns and the first $n-r$ rows of $M(f)$.
	Then $\dim \ker f= t$ if and only if
	\[\det (M_0(f)) =\det (M_1(f)) =\cdots = \det (M_{t-1}(f)) =0, \textrm { and }  \det (M_{t}(f)) \neq 0.\]
\end{theorem}

The cases $q$ odd and $q$ even will be analyzed separately. The next lemma will be crucial in both cases.

\begin{lemma}
\label{mainlemma}
If $N=1$, then $U_b$ is not scattered. If $N\neq 1$ and $U_b$ is not scattered, then
\begin{multline}
\label{Pol_T}
\phi_b(T):=b^{q+1}T^2 +(-b^{q^5+q^4+q^3+q^2+q+1}+2b^{q^4+q^3+q+1}-b^{q^3+1}-b^{q^4+q}+b^{q^5+q^2}) T\\
-b^{q^3+q^4}(b^{q^3+1}-1)^{1+q+q^2}\in \mathbb{F}_{q^6}[T]
\end{multline}
has a root which is a $(1+q+q^2)$-th power in $\F_{q^6}^*$.
\end{lemma}
\begin{proof}
	By Theorem \ref{Th:Sottobicchiere}, $U_b$ is not scattered if and only if there exists $m\in \F_{q^6}$ such that
\begin{equation}
\label{diki}
	\det\left(
	\begin{array}{cccccc}
	m&b&0&0&1&0\\
	0&m^q&b^q&0&0&1\\
	1&0&m^{q^2}&b^{q^2}&0&0\\
	0&1&0&m^{q^3}&b^{q^3}&0\\
	0&0&1&0&m^{q^4}&b^{q^4}\\
	b^{q^5}&0&0&1&0&m^{q^5}\\
	\end{array}
	\right)=\det\left(
	\begin{array}{cccccc}
	b&0&0&1&0\\
	m^q&b^q&0&0&1\\
	0&m^{q^2}&b^{q^2}&0&0\\
	1&0&m^{q^3}&b^{q^3}&0\\
	0&1&0&m^{q^4}&b^{q^4}\\
	\end{array}
	\right)=0.
\end{equation}
The first condition is $G(m,m^q,m^{q^2},m^{q^3},m^{q^4},m^{q^5})=0$, where
\begin{eqnarray*}
G(X,Y,Z,U,V,W)&:=& X Y Z U V W+Y U W+b^{q+1} U V W +X Z V+b^{q^5+1} Z U V \\\nonumber
&&+b^{q^2+q} X V W +b^{q^5+q^4} Y Z U  +b^{q^4+q^3} X Y Z +b^{q^3+q^2}X Y W\\\nonumber
&&-b^{q^5+q^4+q^3+q^2+q+1}+b^{q^5+q^4+q^2+q}+b^{q^4+q^3+q+1}+b^{q^5+q^3+q^2+1}\nonumber\\
&&-b^{q^4+q}-b^{q^3+1}-b^{q^5+q^2}+1,
\end{eqnarray*}
while the second condition is 
\begin{equation}
\label{F0}
F_0(m,m^q,m^{q^2},m^{q^3},m^{q^4},m^{q^5})=0,
\end{equation}
where
	\begin{eqnarray*}
		F_0(X,Y,Z,U,V,W):=-b^{q^4}Y Z U  - b Z U V  + b^{q^4+q^3+q^2+q+1}- b^{q^3+q^2+1}- b^{q^4+q^2+q} + b^{q^2}.
	\end{eqnarray*}
Put
	\begin{eqnarray*}
	F_1(X,Y,Z,U,V,W)&:=&-b^{q^5}Z U V  - b^q U V W  + b^{q^5+q^4+q^3+q^2+q}- b^{q^4+q^3+q}- b^{q^5+q^3+q^2} + b^{q^3},\\
	F_2(X,Y,Z,U,V,W)&:=&-bUVW - b^{q^2}  X V W + b^{q^5+q^4+q^3+q^2+1}- b^{q^5+q^4+q^2}- b^{q^4+q^3+1} + b^{q^4}.
\end{eqnarray*}
Taking $q$-th power and $q^2$-th power of equation \eqref{F0} one obtains 
\begin{equation}
\label{F1}
F_1(m,m^q,m^{q^2},m^{q^3},m^{q^4},m^{q^5})=0
\end{equation}
and
\begin{equation}
\label{F2}
F_2(m,m^q,m^{q^2},m^{q^3},m^{q^4},m^{q^5})=0,
\end{equation}
respectively. Thus \eqref{F0}, \eqref{F1} and \eqref{F2} are pairwise equivalent.

Next, we eliminate $X,Y,Z$ from the above equations. For two multivariate polynomials $f,g\in \overline{\F_q}[X_1,\ldots,X_n]$ denote by $\mathrm{Res}_{X_i}(f,g)$ the resultant of $f$ and $g$ considered as polynomials in $X_i$. Then, with the help of MAGMA, or any other computer algebra system, one can verify that
	\[\mathrm{Res}_Z( \mathrm{Res}_Y ( \mathrm{Res}_X(G,F_0), \mathrm{Res}_X(G,F_1)), 
	\mathrm{Res}_Y ( \mathrm{Res}_X(G,F_0), \mathrm{Res}_X(G,F_2) )  )=-UV \phi_b(UVW)^2.\]
It follows that \eqref{diki} can hold only if $m=0$ or $\phi_b(m^{q^3+q^4+q^5})=0$.

If $m=0$, then \eqref{diki} reads $-(N-1)^{1+q+q^2}=0$ and $b^{q^2}(N-1)^{q+1}=0$; which holds if and only if $N=1$. If $N\neq 1$, then $0$ is not a root of $\phi_b$ and hence $\phi_b(m^{q^3(1+q+q^2)})=0$
only if $\phi_b$ has a root which is a $(1+q+q^2)$-th power in $\F_{q^6}^*$. 
%
\end{proof}

\medskip

\subsection{The $q$ odd case}

In this section we will always assume $q$ odd. First we give sufficient and necessary conditions on $b$ so that $U_b$ is maximum scattered. This condition is equivalent to the one obtained in \cite[Theorem 7.3]{PZ} with different techniques.


The discriminant $\Delta_b$ of $\phi_b(T)$ will play a crucial role in our investigation:
\begin{eqnarray*}
\Delta_b&:=&\left(b^{q^5+q^4+q^3+q^2+q+1}\right)^2  - 2b^{q^5+q^4+2q^3+q^2+q+2}  
-2b^{q^5+2q^4+q^3+q^2+2q+1}\\
 &&- 2b^{2q^5+q^4+q^3+2q^2+q+1}+ b^{2q^3+2} + b^{2q^5+2q^2}+ b^{2q^4+2q}\\
 &&  + 8b^{q^5+q^4+q^3+q^2+q+1}-2b^{q^4+q^3+q+1}-2b^{q^5+q^3+q^2+1}-2b^{q^5+q^4+q^2+q} \\
&=&N^{2(q^2+q+1)}-2(N^{q^2+q+2}+N^{q^2+2q+1}+N^{2q^2+q+1})+N^2+N^{2q}+N^{2q^2}\\
&&+8N^{q^2+q+1}-2(N^{q+1}+N^{q^2+q}+N^{q^2+1}).
\end{eqnarray*}
Note that $\Delta_b$ belongs to $\mathbb{F}_q$.

\begin{proposition}\label{Prop:delta}
There is a root of $\phi_b(T)$ which is a $(q^2+q+1)$-th power in $\mathbb{F}_{q^6}$ if and only if one of the following conditions holds:
\begin{itemize}
\item $\Delta_b=0$,
\item $\Delta_b$ is a square in $\mathbb{F}_q^*$ and $N\in \mathbb{F}_q$,
\item $\Delta_b$ is a non-square in $\mathbb{F}_q$.
\end{itemize}
In all cases, each of the roots of $\phi_b(T)$ are $(1+q+q^2)$-th powers. 
\end{proposition}
\begin{proof}
Let $\delta\in \mathbb{F}_{q^2}$ be such that $\delta^2=\Delta_b$. Consider a root $t$ of $\phi_b(T)$, namely 
$$t=\frac{-B+\delta}{2A},$$
where
\begin{eqnarray}\label{Eq:AB}
A&=&b^{q+1},\\\nonumber
B &=& - b^{q^5+q^4+q^3+q^2+q+1}+2b^{q^4+q^3+q+1}-b^{q^3+1}-b^{q^4+q}+b^{q^5+q^2}.
\end{eqnarray}
We have that $t$ is a $(q^2+q+1)$-th power in $\mathbb{F}_{q^6}$ if and only if $t=0$ or 
$t^{(q-1)(q^3+1)}=1$, equivalently, when $t^{q^4+q}=t^{q^3+1}$, that is, by \eqref{Eq:AB}:
\begin{equation}\label{Eq:t}
A_1\delta^{q^3+1}+A_2(\delta+\delta^{q^3})+A_3 (\delta^{q}+\delta^{q^4})+A_4 \delta^{q^4+q}+A_5=0,
\end{equation}
where
\begin{eqnarray*}
A_1&=&b^{q^5+q^2},\\
A_2&=&b^{2q^5+q^4+q^3+2q^2+q+1} - 2b^{q^5+q^4+q^3+q^2+q+1} + b^{q^5+q^3+q^2+1} + b^{q^5+q^4+q^2+q} -b^{2q^5+2q^2},\\
A_3&=&b^{q^5+q^4+2q^3+q^2+q+2}  + b^{2q^3+2} + 2b^{q^5+q^4+q^3+q^2+q+1} - b^{q^4+q^3+q+1} - b^{q^5+q^3+q^2+1},\\
A_4&=&-b^{q^3+1},\\
A_5&=& -b^{2q^5+2q^4+3q^3+2q^2+2q+3}  + 2b^{q^5+q^4+3q^3+q^2+q+3}- b^{3q^3+3} + b^{3q^5+2q^4+2q^3+3q^2+2q+2}\\
      &&+ 2b^{q^5+2q^4+2q^3+q^2+2q+2} -8b^{q^5+q^4+2q^3+q^2+q+2} + 2b^{q^4+2q^3+q+2} + 3b^{q^5+2q^3+q^2+2} \\
      &&-2b^{2q^5+2q^4+q^3+2q^2+2q+1} - b^{2q^4+q^3+2q+1} -2b^{3q^5+q^4+q^3+3q^2+q+1} + 8b^{2q^5+q^4+q^3+2q^2+q+1}\\
      &&- 3b^{2q^5+q^3+2q^2+1} + b^{q^5+2q^4+q^2+2q}- 2b^{2q^5+q^4+2q^2+q} + b^{3q^5+3q^2}.
\end{eqnarray*}

We distinguish three cases:
\begin{itemize}
\item $\Delta_b=0$. In this case $t=-B/(2A)$ and direct computations show $t^{q^4+q}=t^{q^3+1}$. Hence $t$ is a $(q^2+q+1)$-th power in $\mathbb{F}_{q^6}$.
\item $\Delta_b$ is a square in $\mathbb{F}_q^*$. This means that $\delta^{q^i}=\delta$ for each $i=1,\ldots,5$. Recalling that $\delta^2=\Delta_b$, one gets from \eqref{Eq:t}
\[2\delta(N^{q^2}-N)(-N+N^q-N^{q^2}+N^{1+q+q^2}+\delta)=0.\]
Then either $N\in \mathbb{F}_{q^2}$ or $-N+N^q-N^{q^2}+N^{1+q+q^2}\in \F_q$.
Since $N\in \F_{q^3}$, $N\in \F_q$ follows in both cases. Then $b^{q^5}=b^{q^3+1}/b^{q^2}$ and  $b^{q^4}=b^{q^3+1}/b^q$. Now, $\Delta_b$ reads 
$N^2(N-1)^3(N+3)$. Let $z\in \mathbb{F}_q^*$ be such that $z^2=(N-1)(N+3)$, so that 
$t=(N(N-1)^2+N(N-1)z)/(2b^{q+1})$. It can be verified that $t^{q^4+q}=t^{q^3+1}$ in this case.


\item $\Delta_b$ is not a square in $\mathbb{F}_q^*$. This yields $\delta^{q}=-\delta$ and \eqref{Eq:t} reads
\[-A_1 \delta^2-A_4\delta^2+A_5=-\Delta_b(A_1+A_4)+A_5=0.\]
\end{itemize}

\end{proof}

\begin{theorem}
\label{q_odd}
$U_b$ is  a maximum scattered $\mathbb{F}_q$-subspace of $\mathbb{F}_{q^6} \times \mathbb{F}_{q^6}$ if and only if $\Delta_b$  is a square in $\mathbb{F}_q^*$.
\end{theorem}
\begin{proof}
By Lemma \ref{mainlemma}, if $\phi_b$ has no roots which are $(1+q+q^2)$-th powers in $\F_{q^6}$ then $U_b$ is scattered. On the other hand, if $\phi_b(t)=0$ for some $t=m^{q^3+q^4+q^5}$ where $m\in \F_{q^6}^*$, then 
$m^{q^i+q^{i+1}+q^{i+2}}=t^{q^{i+3}}$ for $i=0,1,\ldots,5$ and also
\[m^{1+q^2+q^4}=\frac{m^{1+q+q^2}m^{q^2+q^3+q^4}}{m^{q+q^2+q^3}}=\frac{t^{q^3+q^5}}{t^{q^4}} \qquad \mbox{ and} \qquad m^{q+q^3+q^5}=\frac{t^{q^4+1}}{t^{q^5}}.\]
In this case, by defining $G$ and $F_0$ as in Lemma \ref{mainlemma} and by using the expressions above, one can substitute the powers of $m$ with the corresponding powers of $t$ as follows:
\begin{eqnarray}
\label{multi1}
\det M_0(m,b) &=&G(m,m^q,m^{q^2},m^{q^3},m^{q^4},m^{q^5})=t^{-q^5-q^4} (b^{q+1} t^{q^5+q^4+1}-b^{q^4+q}t^{q^5+q^4}\\
&&+b^{q^5+1}t^{2q^5+q^4}+b^{q^5+q^4}t^{q^5+2q^4}-b^{q^3+1}t^{q^5+q^4}+b^{q^4+q^3}t^{q^5+q^4+q^3}\nonumber\\
&&+b^{q^4+q^3+q+1}t^{q^5+q^4}-b^{q^5+q^2}t^{q^5+q^4}+b^{q^5+q^4+q^2+q}t^{q^5+q^4}\nonumber\\
&&+b^{q^2+q}t^{q^5+q^4+q}+b^{q^3+q^2}t^{q^2+q^4+q^5}+b^{q^5+q^3+q^2+1}t^{q^5+q^4}\nonumber\\
&&-b^{q^5+q^4+q^3+q^2+q+1}t^{q^5+q^4}+t^{2q^4+1}+t^{q^5+q^4}+t^{2q^5+q^3}+t^{q^5+q^4+q^3+1}),\nonumber\\
\label{multi2}
\det M_1(m,b)&=&F_0(m,m^q,m^{q^2},m^{q^3},m^{q^4},m^{q^5})\\
&=&-b^{q^4+q^2+q}-b^{q^3+q^2+1}+b^{q^4+q^3+q^2+q+1}-b^{q^4} t^{q^4}+b^{q^2}-b t^{q^5}.\nonumber
\end{eqnarray}

By Proposition \ref{Prop:delta}, it is convenient to distinguish four cases:
\begin{itemize}

\item $\Delta_b$ is a square in $\mathbb{F}_q^*$ and $N\notin \mathbb{F}_q$. 
By Proposition \ref{Prop:delta} the roots of $\phi_b(T)$ are not $(q^2+q+1)$-th powers and hence $U_b$ is scattered. 
    \item $\Delta_b=0$. The solution of $\phi_b(T)=0$ is $t=-B/(2A)$, where $A,B$ are as in \eqref{Eq:AB}. By Proposition \ref{Prop:delta}, $t=0$ and hence $N=1$, thus either $U_b$ is not scattered, or we can write $t=m^{q^3+q^4+q^5}$ for some $m\in \F_{q^6}^*$. In the latter case, 
    by substituting $t=-B/(2A)$ in \eqref{multi1} we get zero. 
 Substituting $t=-B/(2A)$ in \eqref{multi2} we get a quantity divisible by $\Delta_b^2$, and hence again zero. It follows that $U_{b}$ is not scattered.

    \item $\Delta_b$ is a square in $\mathbb{F}_q^*$ and $N\in \mathbb{F}_{q}$. In this case a root $t$ of $\phi_b(T)$ can be written as
   \begin{equation}
   \label{faszom}
    t=\frac{N(N-1)^2+N(N-1)z}{2b^{q+1}},
   \end{equation}
    where $z\in \mathbb{F}_q^*$ satisfies $z^2=(N-1)(N+3)$. 
Then substituting \eqref{faszom} in \eqref{multi2} and recalling that $N\in \F_q$, we get
\[\det M_1(m,b)=-b^{-q^5}(N-1)Nz,\]
which is non-zero since $N=1$ would give $\Delta_b=0$. 
It follows that $U_b$ is scattered.
    \item $\Delta_b$ is not a square in $\mathbb{F}_q$. In this case a root $t$ of $\phi_b(T)$ can be written as
    \begin{equation}
    \label{alszom}
    t=\frac{-B+\delta}{2b^{q+1}},
    \end{equation}    
    where $\delta^2=\Delta_b$ and $\delta^q=-\delta$. Note that $t\neq 0$, since then we would have $N=1$ and hence $\Delta_b=0$. Therefore, by Proposition \ref{Prop:delta} we can write $t=m^{q^3+q^4+q^5}$ for some $m\in \F_{q^6}^*$. Then by substituting \eqref{alszom} in \eqref{multi1} and \eqref{multi2} we obtain $\det M_0(m,b) = \det M_1(m,b)=0$ and thus $U_b$ is not scattered.
\end{itemize}
\end{proof}

Next we consider the condition $\Delta_b$ being a square in $\mathbb{F}_q^*$. It will be useful to write $\Delta_b$ as
\begin{equation}\label{Delta_b}
\Delta_b=(N^{q^2+q+1}-N-N^q-N^{q^2})^2+8N^{q^2+q+1}-4(N^{q+1}+N^{q^2+1}+N^{q^2+q}).
\end{equation}


In what follows we will count the exact number of norms $N$ such that $\Delta_b$ is a square in $\mathbb{F}_q^*$.


Note that each $N \in \mathbb{F}_{q^3}\setminus \mathbb{F}_q$ is a root of multiplicity one of the irreducible polynomial 
\begin{equation}
\label{Star0}
T^3-(N+N^q+N^{q^2})T^2+(N^{q+1}+N^{q^2+q}+N^{q^2+1})T-N^{q^2+q+1}\in \F_q[T].
\end{equation}
This fact leads us to consider polynomials
\[F(T):=T^3-ST^2+RT-P\in \F_q[T]\]
satisfying 
\begin{equation}
\label{Eq:Star}
(S-P)^2+8P-4R=U^2  \mbox{ for some } U\in \F_q^*.
\end{equation}

\begin{proposition}
\label{basic}
The number of norms $N\in \F_{q^3} \setminus \F_q$ such that $\Delta_b$ is a square in $\mathbb{F}_q^*$ is three times the number of irreducible polynomials $F(T)=T^3-ST^2+RT-P\in \mathbb{F}_q[T]$ satisfying \eqref{Eq:Star}.
\end{proposition}
\begin{proof}
	It follows from \eqref{Delta_b} and \eqref{Star0}.
\end{proof}

\begin{proposition}
	\label{Finally}
	The number of polynomials $F(T)=T^3-ST^2+RT-P \in \F_q[T]$ satisfying \eqref{Eq:Star} is 
	$(q^3-q^2)/2$.
\end{proposition}
\begin{proof}
For each choice of $S,P\in \mathbb{F}_q$ there are $(q-1)$ pairs $(R,U)$ such that $U\neq 0$ and $(S-P)^2+8P-4R=U^2$. This yields $q^2(q-1)/2$ distinct triples $(S,R,P)\in \mathbb{F}_q^3$ such that $(S-P)^2+8P-4R$ is a square in $\mathbb{F}_q^*$.
\end{proof}

\begin{proposition}
	\label{Prop:Quadruples2}
	The number of pairs $(t,F(T))$, where $F(T)$ satisfies \eqref{Eq:Star} and $t\in \F_q$ is a root of $F(T)$	is exactly $(q^2-q)(q+1)/2$.
\end{proposition}
\begin{proof}
	For each such $t$ we have that $t^3-St^2+Rt-P=0$ and hence $P=t^3-St^2+Rt$.
	From  $(S-P)^2+8P-4R=U^2$ one gets 
	\begin{equation}\label{Eq_t2}
	(S-t^3+St^2-Rt)^2+8(t^3-St^2+Rt)-4R=U^2.
	\end{equation}
\begin{itemize}
	\item If $t\neq 1$, then the above equation defines (in terms of $S$, $R$ and $U$) a hyperbolic quadric which has exactly $2q+1$ ideal points, namely those satisfying 
	\[\left(-R t+S t^2+S-U\right) \left(-R t+S t^2+S+U\right)=0.\] 
	Therefore there exist exactly $(q+1)^2-2q-1=q^2$ triples $(S,R,U)$ satisfying \eqref{Eq_t2}. Among them those with $U=0$ satisfy 
	\[(S-t^3+St^2-Rt)^2+8(t^3-St^2+Rt)-4R=0\]
	and they are $q$, that is the number of the affine $\mathbb{F}_q$-rational points of an irreducible conic with a unique point at infinity obtained when $S=t$ and $R=1+t^2$. This means that the number of distinct pairs $(S,R)\in\F_q^2$ satisfying \eqref{Eq_t2} and such that $(S-t^2+St^2-Rt)^2+8(t^2-St^2+Rt)-4R\neq 0$ is exactly $(q^2-q)/2$ since  $(S,R,-U)$ and $(S,R,U)$ come in pairs. 
	\item Assume $t=1$. In this case $P=1-S+R$ and  Equation \eqref{Eq_t2} reads
	\[(R-2 S-U+3) (R-2 S+U+3)=0.\]
	So, any pair $(S,R)\in \mathbb{F}_q^2$ with $R-2S+3\neq 0$ gives rise to one pair $(1,F(T))$ satisfying our conditions. In this case we obtain $q^2-q$ such pairs.
\end{itemize}
In total, the number of suitable pairs is 
\[(q-1) \frac{q^2-q}{2}+(q^2-q)=\frac{(q^2-q)(q+1)}{2}.\]
\end{proof}

\begin{proposition}
	\label{Prop:Quadruples3}
	The number of pairs $(B,F(T))$, where $F(T)$ satisfies \eqref{Eq:Star} and $F(T)=(T-A)(T-B)(T-B^q)$ for some $A\in \F_q$, $B\in \F_{q^2}$ is exactly 
	\[\frac{1}{2} \left(q^3-2 q^2+2 q+3\right).\]
\end{proposition}
\begin{proof}
	In $(S-P)^2+8P-4R=U^2$ we have $S=A+B+B^q$, $R=AB+AB^q+B^{q+1}$ and $P=AB^{q+1}$.
	Thus one gets 
	\begin{equation}
	\label{Eq_t3}
(B-B^q)^2+A(8B^{q+1}-2(1+B^{q+1})(B+B^q))+A^2(B^{q+1}-1)^2=U^2.
	\end{equation}
When $B^{q+1}\neq 1$,  the discriminant of the left hand side, as a polynomial in $A$, is 
\begin{equation}
\label{disc}
16 (B-1)^{2(q+1)} B^{q+1}.
\end{equation}
	\begin{itemize}
		\item If $B=0$, then \eqref{Eq_t3} reads $A^2=U^2$. In this case $A$ can have $q-1$ values, so there are $q-1$ pairs $(0,F(T))$ with the given conditions. 
		\item If $B=1$, then by \eqref{Eq_t3} $U=0$ and hence $B$ cannot have this value.
		\item If $B\neq 1$ and $B^{q+1}=1$, then the conic in the variables $A$ and $U$ defined by \eqref{Eq_t3} has $q$ affine points and the line $U=0$ contains exactly one of them. 
		Thus the number of solutions for each such $B$ is $(q-1)/2$. There are $q$ possible values for $B$ and thus in total $q(q-1)/2$ pairs $(B,F(T))$ in this case. 
		\item In $\F_{q^2}$, there are $(q^2-1)/2$ values of $B$ such that $B^{q+1}$ is a non-square in $\F_q$. In these cases \eqref{Eq_t3} has no solutions with $U=0$, cf. \eqref{disc}. Since the affine part of the conic in the variables $A$ and $U$, defined by \eqref{Eq_t3}, has $q-1$ affine points, for each such $B$ we get $(q-1)/2$ solutions for $A$ and hence in total $(q^2-1)(q-1)/4$ pairs $(B,F(T))$ with the given conditions.
		\item In $\F_{q^2}$, there are $(q^2-1)/2-(q+1)$ values of $B$ such that $B^{q+1}$ is a square in $\F_q\setminus \{0,1\}$. In these cases \eqref{Eq_t3} has two solutions with $U=0$, and since the affine part of the conic in the variables $A$ and $U$, defined by \eqref{Eq_t3}, has $q-1$ affine points, for each such $B$ we get $(q-3)/2$ solutions for $A$ and hence in total $(q^2-2q-3)(q-3)/4$ pairs $(B,F(T))$ with the given conditions.
	\end{itemize}
	So in total the number of suitable pairs is 
	\[\frac{1}{4} (q-1) \left(q^2-1\right)+\frac{1}{4} (q-3)
	\left(q^2-2 q-3\right)+\frac{1}{2} (q-1) q+q-1=\frac{1}{2} \left(q^3-2 q^2+2 q+3\right).\]
\end{proof}

We are now in the position to prove Theorem \ref{mainthm} when $q$ is odd.
\begin{theorem}\label{Th:Conto}
If $q$ is odd then $|\Gamma|=\frac{1}{2} \left(q^3-q^2-q-3\right)$. 
\end{theorem}
\begin{proof} We will count the exact number of norms $N$ such that $\Delta_b$ is a square of $\F_{q}^*$. 
For $i\in \{0,1,2,3\}$ denote by $\gamma_i$ the number of  polynomials $F(T)=T^3-ST^2+RT-P \in \F_{q}[T]$ with exactly $i$ distinct roots in $\F_q$ and satisfying \eqref{Eq:Star}. 
Consider the pairs $(t,F(T))$, where $F(T)$ is a polynomial satisfying \eqref{Eq:Star} and $t\in \F_q$ is a root of $F(T)$. By Proposition \ref{Prop:Quadruples2}, the number of such pairs is $(q^2-q)(q+1)/2$ and hence also
\begin{equation}
\label{dc0}
\gamma_1+2\gamma_2+3\gamma_3=\frac{(q^2-q)(q+1)}{2}.
\end{equation}
By Proposition \ref{Finally}, the number of polynomials $F(T)$ satisfying \eqref{Eq:Star} is $q^2(q-1)/2$, thus 
\begin{equation}
\label{dc}
\gamma_0=\frac{q^2(q-1)}{2}-\gamma_1-\gamma_2-\gamma_3.
\end{equation}
Denote by $\delta_q$ the number of $N\in \F_q$ such that $\Delta_b$ is a square in $\F_q^*$. 
Note that this number depends on $q$, however, we won't need its exact value. It is easy to see that $\delta_q$ is the same as the number of polynomials $F(T)$ satisfying \eqref{Eq:Star} and with a unique $3$-fold root in $\F_q$. Denote by $A_2$ the number of totally reducible polynomials $F(T)$ over $\F_q$ (polynomials which can be factorized into linear factors over $\F_q$) satisfying \eqref{Eq:Star} and with at most $2$ distinct roots in $\F_q$. Also, denote by $A_1$ the number of reducible polynomials $F(T)$ over $\F_q$ satisfying \eqref{Eq:Star} and with a unique, one-fold root in $\F_q$. Then
\[\gamma_1=A_1+\delta_q,\]
\[\gamma_2=A_2-\delta_q.\]
By Proposition \ref{basic} the number of norms $N$ such that $\Delta_b$ is a square of $\F_q^*$ is, using also \eqref{dc0} and \eqref{dc},
\begin{eqnarray*}
3\gamma_0+\delta_q&=&3\frac{q^2(q-1)}{2}-3\gamma_1-3\gamma_2-3\gamma_3+\delta_q\\
&=&3\frac{q^2(q-1)}{2}-\frac{(q^2-q)(q+1)}{2}-2\gamma_1-\gamma_2+\delta_q\\
&=&3\frac{q^2(q-1)}{2}-\frac{(q^2-q)(q+1)}{2}-2A_1-A_2.
\end{eqnarray*}
In Proposition \ref{Prop:Quadruples3}, $B\in \F_{q^2}\setminus \F_q$ and $B^q$ define the same polynomial $F(T)$ and hence
\[2A_1+A_2=\frac{1}{2} \left(q^3-2 q^2+2 q+3\right),\]
 thus 
\[3\gamma_0+\delta_q=3\frac{q^2(q-1)}{2}-\frac{(q^2-q)(q+1)}{2}-\frac{1}{2} \left(q^3-2 q^2+2 q+3\right)=\frac{1}{2} \left(q^3-q^2-q-3\right).\]
\end{proof}

\medskip

\subsection{The $q$ even case}
Similarly to the $q$ odd case, we will need some preparation. 
Define $N$ and $\phi_b(T)$ in the exact same way as in \eqref{N} and \eqref{Pol_T}, respectively. Since $q$ is even, we have now
%
\[\phi_b(T)=AT^2+BT+C\in \F_{q^6}[T],\]
where
\begin{eqnarray*}
A&=&b^{q+1},\\
B&=&N^{q^2+q+1}+N+N^q+N^{q^2},\\
C&=&A^{q^3}(N+1)^{q^2+q+1}.
\end{eqnarray*}

As before we will assume $b\neq 0$ and hence $A\neq 0$. 

\begin{proposition}\label{Prop_qeven}
There is a root of $\phi_b(T)$ which is a $(q^2+q+1)$-th power in $\mathbb{F}_{q^6}$ if and only if one of the following conditions holds
\begin{itemize}
\item $B=0$,
\item $B\neq 0$, and $N \in \mathbb{F}_q\setminus\{0,1\}$,
\item $B\neq 0$, $N\in  \mathbb{F}_{q^3}\setminus \mathbb{F}_q$,  and $\Tr_{q^3/2}(AC/B^2)=1$.
\end{itemize}
In all cases, each of the roots of $\phi_b(T)$ are $(q^2+q+1)$-th powers.
\end{proposition}
\begin{proof}
First note that a field element $t\in \F_{q^6}$ is a $(q^2+q+1)$-th power if and only if $t^{q^3+1} \in \mathbb{F}_q$. 

Let us deal first with the case $B=0$. In this case the root $t$ of $\phi_b(T)$ satisfies $t^2=C/A$. Therefore
\[t^{2(q^3+1)}=\frac{C^{q^3+1}}{A^{q^3+1}}=(N+1)^{2(q^2+q+1)}\in \mathbb{F}_q\]
and so $t^{q^3+1} \in \mathbb{F}_q.$ 

From now on we assume $B\neq 0$. Then $t$ is a root of $\phi_b(T)$ if and only if $t A/B$ is a root of
\[\psi_b(T):=T^2+T+AC/B^2 \in \F_{q^3}[T].\]
Note that, since $AC/B^2 \in \mathbb{F}_{q^3}$, we have
\[\Tr_{q^6/2}(AC/B^2)=\Tr_{q^6/q^3}\left(\Tr_{q^3/2}(AC/B^2)\right)=0,\]
and therefore the roots of $\psi_b(T)$ (and hence of $\phi_b(T)$) belong to $\mathbb{F}_{q^6}$; see \cite[page 8]{Hirsbook}. 
Also, $B\in \mathbb{F}_q$ and $C/A^{q^3} \in \mathbb{F}_q$.

\begin{itemize}
\item First suppose $\Tr_{q^3/2}(AC/B^2)=1$. Then the roots of $\psi_b(T)$ are some $s, s^{q^3}\in \mathbb{F}_{q^6}\setminus \mathbb{F}_{q^3}$, thus $s^{q^3+1}=AC/B^2$. For a root $t=sB/A$ of $\phi_b(T)$
\[t^{q^3+1}=\left(\frac{sB}{A}\right)^{q^3+1}=\frac{B^2}{A^{q^3+1}}s^{q^3+1}=\frac{C}{A^{q^3}}\in \mathbb{F}_q.\]
\item Suppose now $\Tr_{q^3/2}(AC/B^2)=0$. In this case the roots of $\psi_b(T)$ 
belong to $\F_{q^3}$. We distinguish two cases depending on $N \in \mathbb{F}_q$ or not. Note that $N\notin \{0,1\}$ since $B\neq 0$. 
\begin{itemize}
\item If $N\in \mathbb{F}_q\setminus\{0,1\}$ then $AC/B^2=1/(N+1)\in \mathbb{F}_q$ and hence the roots of $\psi_b(T)$ are some $s, s+1\in \mathbb{F}_q$. So, for a root $t=sB/A$ of $\phi_b(T)$
\[t^{q^3+1}=\left(\frac{sB}{A}\right)^{q^3+1}=\frac{B^2}{N^{2}}s^{2}\in \mathbb{F}_q.\]
\item If $N \in \mathbb{F}_{q^3}\setminus\mathbb{F}_q$ then first observe $AC/B^2\in \mathbb{F}_{q^3}\setminus\mathbb{F}_q$.
Indeed, $AC/B^2=N^{q+1}(N+1)^{q^2+q+1}/(N^{q^2+q+1}+N+N^q+N^{q^2})^2\in \F_q$ would yield $N^{q+1}\in \F_q$ and hence $N\in \F_{q^2} \cap \F_{q^3}=\F_q$, a contradiction. Then also the roots of $\psi_b(T)$ are some $s,s+1\in \mathbb{F}_{q^3}\setminus\mathbb{F}_q$. Suppose, contrary to our claim, that with $t=sB/A$
\begin{equation}
\label{par0}
t^{q^3+1}=\left(\frac{Bs}{A}\right)^{q^3+1}=\frac{B^2}{A^{q^3+1}}s^{2}\in \mathbb{F}_q.
\end{equation}
Since $s^2=s+AC/B^2$, it follows that
\[\frac{B^2}{A^{q^3+1}}(s+AC/B^2)=\frac{B^2}{A^{q^3+1}}s+C/A^{q^3}\in \F_q.\]
Here $C/A^{q^3}\in \F_q$ and hence $\frac{B^2}{A^{q^3+1}}s\in \F_q$. 
Then dividing by \eqref{par0} gives $s\in \F_q$, a contradiction.
\end{itemize}


\end{itemize}
\end{proof}

For a polynomial $f(X)=\sum_i a_i X^i \in \F[X]$ and $\sigma \in \mathrm{Aut}(\F)$ we will denote by  $f^\sigma(X)$ the polynomial $\sum_i a_i^\sigma X^i \in \F[X]$.

\begin{lemma}
	\label{lemmanuovo}
	If $\Tr_{q^3/2}(AC/B^2)=1$ and $t$ is a root of $\phi^{q^4}_b(T)$ 
then $t^q=b^{q^4-1}t+b^{q^2-1}(N+1)^{q+1}$.
\end{lemma}
\begin{proof}
	Substitution of $s:=b^{q^4-1}t+b^{q^2-1}(N+1)^{q+1}$ in $\phi^{q^5}_b(T)$ gives $0$ and hence $s=t^q$ or $s=t^q+B^{q^5}/A^{q^5}=t^q+b^{-1-q^5}B$. In the latter case 
\begin{equation}
\label{tq1}	
	t^q=b^{q^4-1}t+b^{q^2-1}(N+1)^{q+1}+b^{-q^5-1}B,
\end{equation}
	we will show that this is a contradiction.
	Note that 
	\[(tA^{q^4}/B)+(tA^{q^4}/B)^2=(AC)^{q^4}/B^2.\]
Put $e=(tA^{q^4}/B)$ and $E=(AC)^{q^4}/B^2$. Then $e+e^2=E$ and hence 
 $(e+e^2)+(e+e^2)^2+\ldots+(e+e^2)^{q/2}=E+E^2+\ldots+E^{q/2}$, where the left hand side equals $e+e^q$. Since $B$ is in $\F_q$, this can be read as
\begin{equation}
\label{tq2}
	t^qA^{q^5}/B+(tA^{q^4}/B)=E+E^2+\ldots+E^{q/2}.
\end{equation}
Expressing $t^q$ from \eqref{tq2} and comparing with \eqref{tq1} yields
\[b^{q^4-1}t+b^{q^2-1}(N+1)^{q+1}+b^{-q^5-1}B=
tA^{q^4-q^5}+\frac{B}{A^{q^5}}(E+E^2+\ldots+E^{q/2}).\]
Note that $A^{q^4-q^5}=b^{q^4-1}$ and hence we can eliminate $t$ from the equation above.
Then, multiplying by $A^{q^5}/B=b^{1+q^5}/B$ gives
\[\frac{N^{q^2}(N+1)^{q+1}+B}{B}=E+E^2+\ldots+E^{q/2}.\]
The left hand side above is $((N+N^{q^2+q})+((N+N^{q^2+q}))^q)/B$. It follows that taking $\Tr_{q^3/q}$ of both sides yields
\[0=\Tr_{q^3/2}(E)=\Tr_{q^3/2}(AC/B^2)=1,\]
and this contradiction proves the claim.
\end{proof}

\begin{theorem}
\label{q_even}
$U_b$ is a maximum scattered $\mathbb{F}_q$-subspace of $\mathbb{F}_{q^6}\times \F_{q^6}$ if and only if 
\begin{eqnarray}
\label{Condition_qeven}
B \neq 0, \quad N\in \mathbb{F}_{q^3} \setminus\{0,1\},\quad \Tr_{q^3/2}(AC/B^2)=0.
\end{eqnarray}
\end{theorem}
\begin{proof}
We follow the same argument as in Theorem \ref{q_odd}. 
By Lemma \ref{mainlemma}, if $\phi_b$ has no roots which are $(1+q+q^2)$-th powers then $U_b$ is scattered. According to Proposition \ref{Prop_qeven} it is convenient to distinguish the following cases.

\begin{itemize}
\item Suppose that $B=0$.  In this case the solution of $\phi_b(T)=0$ is $t=\sqrt{C/A}\neq 0$ and by Proposition \ref{Prop_qeven} we can find $m\in \F_{q^6}^*$ such that $m^{q^3+q^4+q^5}=\sqrt{C/A}$.
We will show $\det M_0(m,b)=\det M_1(m,b)=0$. By $B=0$, i.e. $\Tr_{q^3/q}(N)=\N_{q^3/q}(N)$, we obtain
\[t^2=c/a=b^{q^4+q^3-q-1}(b^{q^4+q^3+q+1}+b^{q^5+q^4
	+q^2+q}+b^{q^5+q^3+q^2+1}+1).\]
As in \eqref{multi1} and \eqref{multi2}, substitution in 
$\det M_0(m,b)^2=G(m,m^q,m^{q^2},m^{q^3},m^{q^4},m^{q^5})^2$ and in
$\det M_1(m,b)^2=F_0(m,m^q,m^{q^2},m^{q^3},m^{q^4},m^{q^5})^2$, cf. Lemma \ref{mainlemma}, and applying $B=0$ several times, yields $\det M_0(m,b)=\det M_1(m,b)=0$ and hence $U_b$ is not scattered.
\item Suppose that  
\[B\neq 0, \quad N\in \mathbb{F}_{q^3}\setminus \mathbb{F}_q, \quad \Tr_{q^3/2}(AC/B^2)=0.\]
By Proposition  \ref{Prop_qeven} the roots of $\phi_b(T)$ are not $(q^2+q+1)$-th powers in $\mathbb{F}_{q^6}$ and so $U_b$ is scattered by Lemma \ref{mainlemma}.

\item Suppose that 
 $$B\neq 0, \quad N\in \mathbb{F}_{q^3}\setminus \{0,1\}, \quad \Tr_{q^3/2}(AC/B^2)=1.$$
By Proposition  \ref{Prop_qeven} the roots of $\phi_b(T)$ are  $(q^2+q+1)$-th powers in $\mathbb{F}_{q^6}$. Let $\gamma$ be a root of $\phi_b^{q^3}(T)$ and consider $m\in \mathbb{F}_{q^6}^*$ such that $m^{q^2+q+1}=\gamma$. 
Our aim is to prove $G(m,m^q,m^{q^2},m^{q^3},m^{q^4},m^{q^5})=F_0(m,m^q,m^{q^2},m^{q^3},m^{q^4},m^{q^5})=0$.	
To do this, it will be useful to express $\gamma, \gamma^q, \ldots, \gamma^{q^5}$ with respect to only one of these conjugated elements. By Lemma \ref{lemmanuovo},
\begin{equation}
\label{g2}
 \gamma^{q^2}=b^{q^2-1}(N+1)^{q+1} + b^{q^4-1}\gamma^q.
\end{equation}
Taking the $q$-th power in \eqref{g2} and applying again \eqref{g2} yields
\begin{eqnarray*}
\gamma^{q^3}&=&b^{q^3-q}(N^{q^2+q}+N^q+ N^{q^2} + 1) + b^{q^5-q}\gamma^{q^2}\\
&=&(N^{q^2+q}+N^{q+1}+N+N^{q^2}+b^{q^5+q^4}\gamma^q)/b^{q+1}.
\end{eqnarray*}
In a similar way we obtain
\[\gamma^{q^4}=(N^{q^2+q+1} + N+N^q+N^{q^2} + b^{q^5+q^4}\gamma^q)/b^{q^2+q},\]
\[\gamma^{q^5}=(N^{q^2+1}+N+ N^{q^2+q}+N^q +b^{q^5+q^4}\gamma^q)/b^{q^3+q^2},\]
\[\gamma=b^{q-q^3}(N+1)^{1+q^2}+b^{q^5-q^3}\gamma^q.\]
Now, in \eqref{multi1} after substituting $t^{q^i}$ with $\gamma^{q^{i+3}}$, we obtain \[b^{2+q+q^3}G(m,m^q,m^{q^2},m^{q^3},m^{q^4},m^{q^5})=\]
\[\phi_b^{q^4}(\gamma^q)(\phi_b^{q^4}(\gamma^q)+\gamma^q N^{q^2}(N+N^q)+b^{q+q^2}(N+1)^{q+1}(N+N^{q^2}))=0.\]
Similarly, substitution in \eqref{multi2} yields $F_0(m,m^q,m^{q^2},m^{q^3},m^{q^4},m^{q^5})=0$ and hence $U_b$ is not scattered.

\item Suppose that 
\[B\neq 0, \quad N\in \mathbb{F}_{q}\setminus \{0,1\}, \quad \Tr_{q^3/2}(AC/B^2)=0.\]
By Proposition  \ref{Prop_qeven} the roots of $\phi_b(T)$ are  $(q^2+q+1)$-powers in $\mathbb{F}_{q^6}$. Let $t=m^{q^3+q^4+q^5}\in \F_{q^6}^*$ be a root of $\phi_b(T)$. Then, as in Proposition \ref{Prop_qeven}, there exists $s\in \F_q$, a root of $\psi_b(T)$, such that $tA=sB \in \F_q$. Then $t^{q^i}=tA/A^{q^i}$ for $i=0,1,\ldots,5$. 
Substituting $t^{q^4}=t b^{q+1-q^5-q^4}$ and $t^{q^5}=t b^{q-q^5}$ in 
$\det M_1(m,b)=F_0(m,m^q,m^{q^2},m^{q^3},m^{q^4},m^{q^5})$ gives $b^{q^2}(N+1)^2$ which cannot be zero and hence $U_b$ is scattered.
\end{itemize}
\end{proof}


From now on, our aim is to calculate the number of norms $N$ so that  Conditions \eqref{Condition_qeven} are satisfied. One can easily check that
\begin{multline}
\label{vegefele}
\Tr_{q^3/2}(AC/B^2)=\Tr_{q/2} \left( \Tr_{q^3/q}(AC/B^2) \right)=\\
\Tr_{q/2} \left( \frac{(N^{q+1}+N^{q^2+1}+N^{q^2+q})(N+N^q+N^{q^2}+N^{q+1}+N^{q^2+1}+N^{q^2+q}+N^{q^2+q+1}+1)}
{(N^{q^2+q+1}+N+N^q+N^{q^2})^2}\right).
\end{multline}
Each $N\in \F_{q^3} \setminus \F_q$ is a root of multiplicity one of the irreducible polynomial
\[T^3+(N+N^q+N^{q^2})T^2+(N^{q+1}+N^{q^2+q}+N^{q^2+1})T+N^{q^2+q+1} \in \F_q[T].\]
This leads us to consider polynomials
\begin{equation}
\label{vegefele1}
F(T):=T^3+ST^2+RT+P\in \mathbb{F}_q[T],
\end{equation}
satisfying
\begin{equation}
\label{vegefele2}
P \neq S \quad \mbox{ and } \quad \Tr_{q/2} \left( \frac{R(S+P+R+1)}{(P+S)^2}\right)=0,
\end{equation}
or equivalently, polynomials of the form \eqref{vegefele1} so that
\begin{equation}
\label{Eq:Star_even}
P\neq S \quad \mbox{ and } \quad U^2+U+\frac{R(S+P+R+1)}{(P+S)^2}\in \mathbb{F}_q[U] \textrm{ has two roots in } \mathbb{F}_q.
 \end{equation}
Note that if $P=N^{1+q+q^2}$ and $S=N+N^q+N^{q^2}$, then $P\neq S$ yields $N\notin \{0,1\}$. 
We proceed similarly to the case $q$ odd. After rearranging \eqref{Eq:Star_even}, it is clear that we have to study $4$-tuples $(U,S,P,R)$ satisfying
\begin{equation}
\label{kiemel}
P \neq S \quad \mbox{ and } \quad (P+S)^2 U^2+(P+S)^2 U + R(S+P+R+1) =0.
\end{equation}
Note that $(u,s,p,r)$ satisfies the equation above if and only if $(u+1,s,p,r)$ satisfies it and for fixed $s,p,r$ there cannot be further solutions.

\begin{proposition}
	\label{Finally_even}
	The number of polynomials $F(T)=T^3+ST^2+RT+P\in \F_q[T]$ satisfying \eqref{Eq:Star_even} is $(q^3-q^2)/2$.
\end{proposition}
\begin{proof}
For each choice of the pair $(P,S)$, $P\neq S$, there are $q$ pairs $(R,U)$ such that
\eqref{kiemel} holds, these pairs correspond to affine points of an irreducible conic. The fact that it is irreducible follows easily by considering partial 
derivatives of the homogeneous equation for this conic. It turns out that at the points of this curve, the three partial derivatives cannot vanish at the same time, and hence the curve is non-singular.  
Then the number of quadruples $(U,S,R,P)$ satisfying \eqref{kiemel} is $q^3-q^2$. 
The number of polynomials $F(T)$ satisfying \eqref{Eq:Star_even} is the number of distinct triples $(S,R,P)\in \mathbb{F}_q^3$ such that \eqref{kiemel} is satisfied, which is
$(q^3-q^2)/2$. 
\end{proof}

\begin{proposition}
	\label{Prop:Quadruples2_even}
	The number of pairs $(t,F(T))$, where $F(T)$ satisfies \eqref{Eq:Star_even} and $t\in \F_q$ is a root of $F(T)$ is exactly $(q^3-q)/2$.
\end{proposition}
\begin{proof}
	We will frequently use $P=t^3+St^2+Rt$ to eliminate $P$. 
	
	If $t=0$ then $P=0$ and from \eqref{kiemel} one gets $(U^2+U)S^2+R(S+R+1)=0$. 
	\begin{itemize}
	\item If $U\notin\{0,1\}$ then there are precisely $q-1$ pairs $(S,R)\in\mathbb{F}_q^2$ such that $(U^2+U)S^2+R(S+R+1)=0$, two of them with $S=0$. So in total $q-3$ pairs such that $P\neq S$. 
	\item If $U=0$ or $U=1$ then there are $2q-1$ pairs $(S,R)$, two of them with $S=0$. So in total $2q-3$ pairs such that $P\neq S$. 
	\end{itemize}
	Summing up, if $t=0$ then the number of 
	$5$-tuples $(t,U,S,P,R)$ such that  $F(T)$ satisfies \eqref{Eq:Star_even} is
	$(q-2)(q-3)+2(2q-3)=q^2-q$.
	
	If $t=1$, then $P=1+S+R$ and from \eqref{kiemel} one gets $U(U+1)(R+1)^2=0$. 
	\begin{itemize}
		\item If $U\notin\{0,1\}$ then $R=1$ and hence $P=S$, so there are no $5$-tuples with the required properties. 
		\item If $U\in \{0,1\}$ then there are $q^2-q$ pairs $(S,P)$ with $P\neq S$, and each of these pairs uniquely determine $R$. 
	\end{itemize}
	Summing up, if $t=1$ then the number of $5$-tuples $(t,U,S,P,R)$ such that  $F(T)$ satisfies \eqref{Eq:Star_even} is $2(q^2-q)$.
	
Suppose now $t\notin \{0,1\}$.
	\begin{itemize}
	\item If $U \in \{0,1\}$ then the equation in \eqref{kiemel} reads $(t + 1)R(t^2+t+1+S(t+1)+R)=0$ which is satisfied exactly by $2q-1$ pairs $(R,S)$, corresponding to affine points of two intersecting affine line. We have $P=S$ if and only if $t^3+S(t^2+1)+Rt=0$ and hence when 
	$(S,R)\in \{(t,1),(t^3/(t^2+1),0)\}$. Thus the number of $5$-tuples $(t,U,S,P,R)$ such that $F(T)$ satisfies \eqref{Eq:Star_even} is $2(q-2)(2q-3)$. 
	\item If $U \notin\{0,1\}$ then the equation in \eqref{kiemel} defines the affine part of a conic $\cC$ in the variables $S=:S'/Z$ and $R=:R'/Z$. The homogenous equation in $S',R',Z$ of $\cC$ is
\begin{eqnarray*}
R'^2(1+t+t^2U+t^2U^2)+S'^2(U+t^4U+U^2+t^4U^2)\\+R'S'(1+t^2)+Z^2(t^6U+t^6U^2)+R'Z(1+t^3).
\end{eqnarray*}
Its partial derivatives are zero only at the point $(1+t+t^2,0,1+t)$. Substituting these values in the equation above yields $(1+t^2)U(U+1)$ which is non-zero and hence $\cC$ is irreducible.

	Substituting $Z=0$ yields
	\[(S'(U t^2+U)+R'(Ut+1))(S'(Ut^2+U+t^2+1)+R'(Ut+t+1))=0,\]
	and hence $\cC$ has two distinct  $\mathbb{F}_q$-rational points at infinity.
	
It remains to calculate the number of affine points of $\cC$ on the affine line 
$t^3+S(t^2+1)+R=0$ (whose points correspond to solutions with $P=S$). 
Substitution of $R$ in \eqref{kiemel} gives
	\[(t^2+1)(t^3+S(t^2+1))(t^3 U^2 + t^3 U + t^2+1 + S(t^2 U^2 + t^2U+t+U^2+U))=0.\]
It is easy to see that the two solutions in $S$ are distinct and hence in this case 
the number of $5$-tuples $(t,U,S,P,R)$ such that $F(T)$ satisfies \eqref{Eq:Star_even} is $(q-2)(q-2)(q-3)$. 
	\end{itemize}

The total number of pairs $(t,F(T))$ is the number of 
$4$-tuples $(t,S,P,R)$ such that $F(T)$ satisfies \eqref{Eq:Star_even}, which is half of the number of $5$-tuples $(t,U,S,P,R)$ with the same property, i.e. 
\[\frac{1}{2}((q^2-q) +2(q^2-q)+2(q-2)(2q-3)+(q-2)(q-2)(q-3)))=\frac{q^3-q}{2}.\]
\end{proof}

\begin{lemma}
\label{FF}
For $\alpha, \beta, \gamma \in \F_q$, $\alpha\neq 0$, the number of pairs $(X,Y)\in \F_q^2$, $X\neq \gamma$, such that $Y^2+Y=\alpha (X+1)(X+\beta)/(X^2+\gamma^2)$ and
\begin{equation}
\label{FF_2}
\alpha(\beta^2+1)\neq (\gamma+\beta)(\gamma+1),
\end{equation}
is either $q$ or $q-2$ depending on $\Tr_{q/2}(\alpha)=1$ or $\Tr_{q/2}(\alpha)=0$, respectively.
\end{lemma}
\begin{proof}
	First note that
	\[\alpha \frac{(X+1)(X+\beta)}{(X^2+\gamma^2)}=\alpha\left(1+\frac{\beta+1}{X+\gamma}+\frac{(\gamma+\beta)(\gamma+1)}{X^2+\gamma^2}\right).\]
	Introduce the variable $W:=1/(X+\gamma)$. Then our equation becomes
	\[Y^2+Y=\alpha+\alpha(\beta+1)W+\alpha(\gamma+\beta)(\gamma+1)W^2,\]
	which can be considered as the affine part of a conic $\cC$ in the variables 
$W=:W'/Z$ and $Y=:Y'/Z$. The partial derivatives of the homogenous equation of $\cC$ are zero only at the point $(a(1+b),1,0)$, which is not a point of $\cC$ because of the assumption \eqref{FF_2}. It is also easy to see that $\cC$ has a unique point at infinity which is $\mathbb{F}_q$-rational. 
	 The number of affine $\mathbb{F}_q$-rational points of $\cC$ with $W=0$ is $2$ when $\Tr_{q/2}(\alpha)=0$ and zero otherwise. For each solution $(W,Y)$ with $W\neq 0$ there corresponds a unique solution $(X,Y)$ with $X=1/W+\gamma$ and hence the result follows.
\end{proof}

\begin{proposition}
	\label{Prop:Quadruples3_even}
	The number of polynomials $F(T)$, where $F(T)$ satisfies \eqref{Eq:Star_even} and $F(T)=(T+K)(T+L)(T+L^q)$ for some $K\in \F_q$, $L\in \F_{q^2}\setminus \F_q$, is exactly 
	\[\frac{q^3-3q^2+4q}{4}.\]
\end{proposition}
\begin{proof}
Since $F(T)=T^3+ST^2+RT+P$, we have $S=K+L+L^q$, $R=KL+KL^q+L^{q+1}$ and $P=KL^{q+1}$. 
Note that $P\neq S$ is equivalent to $K (L^{q+1}+1)\neq (L+L^q)$ and in this case the equation in \eqref{Eq:Star_even} reads
	\begin{equation}
	\label{Eq_t3_even}
U^2+U=\frac{(L + 1)^{q+1}(K + 1)((L+L^q)K + L^{q+1})}{((L^{q+1}+1)K+ L + L^q)^2}.
	\end{equation}
First suppose $L^{q+1}=1$. Since $L\notin \F_q$, in this case $S\neq P$ holds and 
	\eqref{Eq_t3_even} reads 
	\[( L + L^q)^2(U^2+U)=(K + 1)((L+L^q)K + 1).\]
	Each solution $(K,U)$ corresponds to an affine point of an irreducible conic with a unique ideal point and hence there are $q$ such pairs which yields $q/2$ possible values for $K$. Since $L$ can be chosen in $q$ different ways, in total we have $q^2/2$ different  pairs $(K,L)$ such that $F(T)$ satisfies \eqref{Eq:Star_even}.
	
From now on, assume $L^{q+1}\neq 1$. Then, since $L\notin \F_q$, we can define  $\alpha=\frac{(L+1)^{q+1}(L+L^q)}{(L^{q+1}+1)^2}$, $\beta=\frac{L^{q+1}}{L^q+L}$ and $\gamma =\frac{L+L^q}{L^{q+1}+1}$ to apply Lemma \ref{FF}. Condition \eqref{FF_2} is satisfied and clearly $\alpha\neq 0$. 
Note that $S\neq P$ is equivalent to $\gamma \neq K$ and hence to determine
the number of pairs $(K,U)\in \F_q^2$ satisfying \eqref{Eq_t3_even} we can apply Lemma \ref{FF}. Note that
	\[\Tr_{q/2}(\alpha)=
	\Tr_{q/2}\left( \frac{(L^{q+1}+1+L^q+L)(L+L^q)}{(L^{q+1}+1)^2}\right)=\Tr_{q/2}\left(\frac{L+L^q}{L^{q+1}+1}+ \frac{(L+L^q)^2}{(L^{q+1}+1)^2}\right)=0.\]
It follows that for each $L$ there are $q-2$ pairs $(K,U)$ satisfying \eqref{Eq_t3_even}  and hence $(q-2)/2$ possible values for $K$. There are $q^2-2q$ choices for $L$ in $\F_{q^2}\setminus \F_q$ such that $L^{q+1}\neq 1$. In total we have $(q-2)(q^2-2q)/2$ different  pairs $(K,L)$ such that $F(T)$ satisfies \eqref{Eq:Star_even}.

Since $(K,L,L^q)$ and $(K,L^q,L)$ give rise to the same polynomial $F(T)=(T+K)(T+L)(T+L^q)$, in total the number of polynomials $F(T)$ satisfying \eqref{Eq:Star_even} is
\[\frac{1}{2}\left(\frac{q^2}{2}+\frac{(q^2-2q)(q-2)}{2}\right)=\frac{q^3-3q^2+4q}{4}.\]
\end{proof}

\begin{proposition}
	\label{Prop:Quadruples4_even}
	The number of polynomials $F(T)=(T+K)^3$, $K\in \mathbb{F}_q$, satisfying \eqref{Eq:Star_even} is  
	\[ \frac{q-2\gcd(n,2)}{2}.\]
\end{proposition}
\begin{proof}
	Since $F(T)=T^3+ST^2+RT+P$, we have $S=K$, $R=K^2$, and $P=K^3$. 
	Note that $K\notin \{0,1\}$ otherwise $S=P$. Then \eqref{Eq:Star_even} reads
	\begin{equation}
U^2+U=\frac{1}{K+1},
	\end{equation}
which has two roots in $U$ for $(q-4)/2$ distinct values of $K\in \F_q \setminus \{0,1\}$ if $n$ is even, and $(q-2)/2$ distinct values of $K\in \F_q \setminus \{0,1\}$ if $n$ is odd.
\end{proof}

\begin{proposition}
	\label{Prop:Quadruples5_even}
	The number of polynomials $F(T)=(T+K)(T+L)^2$, $K, L\in \mathbb{F}_q$, $K\neq L$, which satisfy \eqref{Eq:Star_even} is  
	\[\frac{q^2-3q+2+2\gcd(n,2)}{2}.\]
\end{proposition}
\begin{proof}
Since $F(T)=T^3+ST^2+RT+P$, we have $S=K$, $R=L^2$, and $P=KL^2$. 
Note that $K\neq 0$ and $L\neq 1$ otherwise $S=P$. Then from the equation in \eqref{kiemel} one gets
	\begin{equation}
	\label{minnya}
K^2(L^2+1)(U^2+U)=L^2(K+1).
	\end{equation}
If $U\in \{0,1\}$, then $K=1$ and $L\in \F_q \setminus \{1\}$, or $L=0$ and $K\in \F_q \setminus \{0\}$. The pair $(K,L)=(1,0)$ is counted twice, so in total in this case there are 
$2q-3$ pairs $(K,L)$ satisfying $K\neq 0$, $L\neq 1$ and \eqref{minnya}.

From now on suppose $U\notin \{0,1\}$. We have $L\neq 1$ so we can put $V:=L^2/(L^2+1)$. Then $V\neq 1$ and $L^2=V/(V+1)$. Each solution $(K,V)$ of $K^2(U^2+U)=V(K+1)$ corresponds to an affine point of an irreducible conic with $2$ $\mathbb{F}_q$-rational points at infinity.
There is one affine solution with $K=0$ and $2$ affine solutions with $V=1$.
It follows that for each $U$ there are $q-4$ pairs of solutions with $K\neq 0$ and $V\neq 1$, hence the same number of pairs of solutions $(K,L)$ with $K\neq 0$ and $L\neq 1$.

For $U$ and $U+1$ there corresponds the same pair of solution $(K,L)$ and hence in total 
there are 
\[\frac{2(2q-3)+(q-2)(q-4)}{2}=\frac{q^2-2q+2}{2}\]
suitable pairs $(K,L)$. By Proposition \ref{Prop:Quadruples5_even},  $(q-2\gcd(n,2))/2$ of them have $K=L$. 
After subtracting this number we obtain the number of polynomials $F(T)$ satisfying \eqref{Eq:Star_even}.
\end{proof}

\begin{theorem}
\label{qeven:Conto}
If $q$ is even then $|\Gamma|=\frac12 (q^3-q^2-q-2)$.
\end{theorem}
\begin{proof}
Finally, we will count the number of norms satisfying \eqref{Condition_qeven}. 
For $i\in \{0,1,2,3\}$, denote by $\gamma_i$ the number of polynomials 
$F(T)=T^3+ST^2+RT+P \in \F_q[T]$ with exactly $i$ distinct roots in $\F_q$ and satisfying 
\eqref{Eq:Star_even}. 
By Proposition \ref{Prop:Quadruples2_even},
\[\gamma_1+2\gamma_2+3\gamma_3=\frac{q^3-q}{2}.\]
By Propositions \ref{Prop:Quadruples3_even} and \ref{Prop:Quadruples4_even},
\[\gamma_1=\frac{q^3-3q^2+4q}{4}+\frac{q-2 \gcd(n,2)}{2},\]
and by Proposition \ref{Prop:Quadruples5_even},
\[\gamma_2=\frac{q^2-3q+2+2\gcd(n,2)}{2}.\]
Combining these results, we obtain
\[\gamma_3=\frac{q^3 - q^2 +4q -4\gcd(n,2) - 8}{12}.\]
By Proposition \ref{Finally_even} the number of polynomials satisfying 
\eqref{Eq:Star_even} is $(q^3-q^2)/2$ and hence
\[\gamma_0=\frac{q^3-q^2}{2}-(\gamma_1+\gamma_2+\gamma_3)=\frac{q^3 - q^2 - 2q + 2\gcd(n,2) - 2}{6}.\]

The number of values of $N$ such that \eqref{Condition_qeven} is satisfied equals 
$3$ times the number of irreducible polynomials $F(T)$ satisfying \eqref{Eq:Star_even} 
(corresponding to solutions with $N\in \F_{q^3} \setminus \F_q$) 
plus the number of polynomials $F(T)$ with a unique $3$-fold root in $\F_q$ and satisfying \eqref{Eq:Star_even} (corresponding to solutions with $N\in \F_q$).
Thus this number is 
\[3\gamma_0 + \frac{q-2\gcd(n,2)}{2}=\frac{q^3 - q^2 - q - 2}{2}.\]
\end{proof}

%
%


\section{Acknowledgements}

The first author was partially supported by the Ministry for Education, University and Research of Italy (MIUR) and by the Italian National Group for Algebraic and Geometric Structures and their
	Applications (GNSAGA-INdAM). The second author was partially supported by the J\'anos Bolyai Research Scholarship of the Hungarian Academy of Sciences and by the National Research, Development and Innovation Office - NKFIH under the grants PD 132463 and K 124950.

\end{document}